\documentclass[twoside]{amsart}
\usepackage{latexsym}
\usepackage{amssymb,amsmath,amsopn,setspace}
\usepackage[dvips]{graphicx}   
\usepackage{color,epsfig}      

\newcommand{\mathsym}[1]{{}}
\newcommand{\unicode}[1]{{}}

\newtheorem{theorem}{Theorem}[]
\newtheorem{corollary}[]{Corollary}[]

\newtheorem{prop}[]{Proposition}[]

\theoremstyle{definition}

\newtheorem{remark}[]{Remark}[]
\newtheorem{example}[]{Example}[]

\numberwithin{equation}{section}

\newcommand{\R}{\mathbb{R}}

\newcommand{\Sph}{\mathbb{S}}

\begin{document}

\title[On the icosahedron inequality]{On the icosahedron inequality of L\'aszl\'o Fejes-T\'oth}
\author[\'A. G.Horv\'ath]{\'Akos G.Horv\'ath}
\date{2014 Nov.}

\address{\'A. G.Horv\'ath, Dept. of Geometry, Budapest University of Technology,
Egry J\'ozsef u. 1., Budapest, Hungary, 1111}
\email{ghorvath@math.bme.hu}


\begin{abstract}
In this paper we deal with the problem to find the maximal volume polyhedra with a prescribed property and inscribed in the unit sphere. We generalize those inequality (called by \emph{icosahedron inequality}) of L. Fejes-T\'oth of which an interesting consequence the fact that regular icosahedron has maximal volume in the class of the polyhedra with twelve vertices and inscribed in the unit sphere. We give an upper bound for the volume of such star-shaped (with respect to the origin) triangular polyhedra which we know the maximal edge lengths of its faces, respectively. As a further consequence of the generalized inequality we prove the conjecture which states that the maximal volume polyhedron spanned by the vertices of two regular simplices with common centroid is the cube.
\end{abstract}

\maketitle

\section{Introduction}
To find the maximal volume polyhedra in $\R^3$ with a given number of vertices and inscribed in the unit sphere, was first mentioned in \cite{ftl} in 1964. A systematic investigation of this question was started with the paper \cite{bermanhanes} of Berman and Hanes in 1970, who found a necessary condition for optimal polyhedra, and determined those with $n \leq 8$ vertices. The same problem was examined in \cite{M02}, where the author presented the results of a computer-aided search for optimal polyhedra with $4 \leq n \leq 30$ vertices. Nevertheless, according to our knowledge, this question, which is listed in both research problem books \cite{BMP05} and \cite{croft}, is still open for polyhedra with $n > 8$ vertices apart from the fortunate case of $n=12$ when the solution is the regular icosahedron. In \cite{gholangi} the authors investigated this problem for polytopes in arbitrary dimensions. By generalizing the methods of \cite{bermanhanes}, presented a necessary condition for the optimality of a polytope.
The author found the maximum volume polytopes in $\R^d$, inscribed in the unit sphere $\Sph^{d-1}$, with $n=d+2$ vertices. For $n=d+3$ vertices, found the maximum for $d$ odd, over the family of all polytopes, and for $d$ even, over the family of not cyclic polytopes.

One of the most important tools of $3$-dimensional investigations is the result of L. Fejes-T\'oth on the volume bounds of the polyhedra inscribed in the unit sphere (formula (2) on p. 263 in \cite{ftl}). For triangular polyhedra it can be simplified into another one (see  p.264 in \cite{ftl})  which we call \emph{icosahedron inequality}. The calling motivated by the fact, that this inequality implies the case of $n=12$ points when the unique solution is the icosahedron.

The aim of this paper is to give similar inequalities for such cases when it need to take into consideration certain (other than the number of vertices) prescribed information on the examined class of polytopes inscribed in the unit sphere. In Section 3 we generalize the icosahedron inequality for such triangular bodies which faces has given the lengths of the maximal edges, respectively (Prop. 2, Prop. 3, Theorem 1). Our extracted formula is valid not only for convex polyhedra but star-shaped with respect to origin polyhedra, too. (Theorem 1).
As a further consequence of the generalized inequality we prove the conjecture which states that the maximal volume polyhedron spanned by the vertices of two regular simplices with common centroid is the cube. This conjecture was raised and proved partially in \cite{gho 1} and inspired some other examinations on the volume of the convex hull of simplices \cite{gho 2}.

In Section 4 we consider the general (non-triangular) case and prove some inequalities on it, too. Finally, we give the source file (written by Mathematica 10) of the symbolic and numerous calculations of the proof of Theorem 1 in the last section.

\section{Preliminaries}

Since we use some of the important steps of the proof of these inequalities we collect them in a separate proposition. Let $a(P)$ be the area of a convex $p$-gon $P$ lying in the unit sphere, $\tau(P)$ the (spherical) area of the central projection of $P$ upon the unit sphere, and $v(P)$ the volume of the pyramid of base $P$ and apex $O$ which is the centre of the unit sphere. Let denote $U(\tau(P),p)$ the maximum of $v(P)$ for a given pair of values $p$ and $\tau(P)$.
\begin{prop}[\cite{ftl}] With the above notation we have the following statements on $U(\tau(P),p)$.
\begin{enumerate}
\item For given values of $p$ and $\tau$ the volume $v$ attains its maximum $U(\tau, p)$ if $t$ is a regular $p$-gon.
\item For general $p\geq 3$ we have
\begin{equation}
U(\tau,p)= \frac{p}{3}\cos ^2\frac{\pi}{p}\tan\frac{2\pi-\tau}{2p}\left(1-\cot ^2\frac{\pi}{p}\tan ^2\frac{2\pi-\tau}{2p}\right),
\end{equation}
implying that
\begin{equation}
U(\tau,3)=\frac{1}{4}\tan\frac{2\pi-\tau}{6}\left(1-\frac{1}{3}\tan ^2\frac{2\pi-\tau}{6}\right),
\end{equation}

\item The function $U(\tau, p)$ concave on the domain determined by the inequalities $0<\tau\leq \pi$, $p\geq 3$.

\item If $V$ denotes the volume, $R$ the circumradius of a convex polyhedron having $f$ faces, $v$ vertices and $e$ edges, then
\begin{equation}
V\leq \frac{2e}{3}\cos^2\frac{\pi f}{2e}\cot\frac{\pi v}{2e}\left(1-\cot^2\frac{\pi f}{2e}\cot^2\frac{\pi v}{2e}\right)R^3.
\end{equation}
Equality holds only for regular polyhedra.

\end{enumerate}
\end{prop}

\begin{remark} A polyhedron with a given number $n$ of vertices is always the limiting figure of a triangular polyhedron with $n$ vertices, hence, introducing the notation
$$
\omega_n=\frac{n}{n-2}\frac{\pi}{6}
$$
we have the following inequality
\begin{equation}
V\leq \frac{1}{6}(n-2)\cot \omega_n(3-\cot^2\omega _n)R^3.
\end{equation}
Equality holds in (2.4) only for the regular tetrahedron, octahedron and icosahedron ($n=4,6,12$).
\end{remark}

\section{Inequalities on the volume of a facial tetrahedron}

Our first proposition is a rewriting of (1) of Proposition 1 in a more general form when $p=3$. If $A,B,C$ are three points on the unit sphere we can say two triangles with these vertices, one of the corresponding spherical triangle and the second one the rectilineal triangle with these vertices, respectively. Both of them are denoted by $ABC$. The angles of the rectilineal triangle are the half of the angles between those radius of the circumscribed circle which connect the center $K$ of the rectilineal triangle $ABC$ to the vertices $A,B,C$. Since $K$ is also the foot of the altitude of the tetrahedron with base $ABC$ and apex $O$, hence the angles $\alpha_A$, $\alpha_B$ and $\alpha_C$ of the rectilineal triangle $ABC$, play an important role in our investigations, we refer to them as the \emph{central angle} of the spherical edges $BC$, $AC$ and $AB$, respectively. We call the tetrahedron $ABCO$ the \emph{facial tetrahedron} with base $ABC$ and apex $O$.

\begin{prop}
Let $ABC$ be a triangle inscribed into the unit sphere. Then there is an isosceles triangle $A'B'C'$ inscribed into the unit sphere with the following properties:
\begin{itemize}
\item the greatest central angles and also the spherical areas of the two triangles are equal to each other, respectively;
\item the volume of the facial tetrahedron with base $A'B'C'$  is greater than or equal to the volume of the facial tetrahedron with base $ABC$.
\end{itemize}
\end{prop}

\begin{figure}[ht]
\includegraphics[scale=0.8]{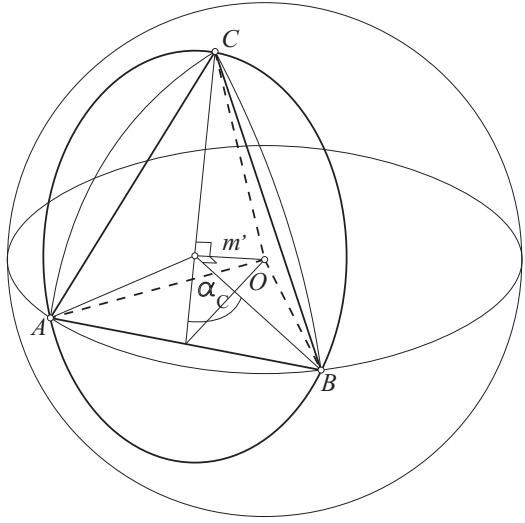}
\caption[]{Facial, rectilineal and spherical simplices, respectively.}
\end{figure}

\begin{proof}
Assume first that the triangle $ABC$ contains the centre $K$ of its circumscribed circle. Let denote by $K'$ the central projection of $K$ onto the unit sphere. The angles $2\alpha_A$ and $\beta_A$ the spherical angles of the triangle $K'BC$ at $K'$ and $B$ (or $C$), respectively. Then the area of the triangle
$KBC$ is equal to
$$
a(KBC)=\Delta(\alpha_A,\beta_A)=\frac{1}{2}\sin 2\alpha_A\sin ^2 K'OB\angle=\frac{1}{2}\sin 2\alpha_A\left(1-\cot^2\alpha_A\cot^2\beta_A\right).
$$
On the domain
$$
0\leq \alpha \leq \frac{\pi}{2}, \quad 0\leq \beta \leq \frac{\pi}{2}, \quad \alpha+\beta \geq \frac{\pi}{2}
$$
it is a concave function of two variables (see p.267 in \cite{ftl}). Hence
$$
a(ABC)=\Delta(\alpha_A,\beta_A)+\Delta(\alpha_B,\beta_B)+\Delta(\alpha_C,\beta_C)\leq
$$
$$
\leq 2\Delta\left(\frac{\alpha_A+\alpha_B}{2},\frac{\beta_A+\beta_B}{2}\right)+\Delta(\alpha_C,\beta_C)=a(A'B'C'),
$$
where the value on the right hand side of the inequality above is the area of the isosceles triangle $A'B'C'$. (We note that the central projections upon the sphere of the two triangles have the same spherical excess $a(ABC)=a(A'B'C')=2(\beta_A+\beta_B+\beta_C)-\pi$.)

Compare now the altitudes $m$ and $m'$ of the pyramids based on the two triangles, respectively. The spherical area of the first triangle is
$$
\tau=2(\beta_A+\beta_B+\beta_C)-\pi=2\pi+\left(2(\beta_A+\beta_B+\beta_C)-3\pi\right)=2\pi+
$$
$$
+2\left(\tan^{-1}\left(\tan\left(\beta_A-\frac{\pi}{2}\right)\right)+ \tan^{-1}\left(\tan\left(\beta_B-\frac{\pi}{2}\right)\right)+\right.
$$
$$
\left.+\tan^{-1}\left(\tan\left(\beta_C-\frac{\pi}{2}\right)\right)\right)=2\pi-2\left(\tan^{-1}\left(m\tan\alpha_A\right)+\right.
$$
$$
\left.+\tan^{-1}\left(m\tan\alpha_B\right)+\tan^{-1}\left(m\tan\alpha_C\right)\right).
$$
By the convexity (see e.g. p. 229 in \cite{ftl}) of $\tan^{-1}\left(m\tan\alpha_A\right)$ we get that
$$
\tau\leq 2\pi-2\left(2\tan^{-1}\left(m\tan\frac{\alpha_A+\alpha_B}{2}\right)+\tan^{-1}\left(m\tan\alpha_C\right)\right).
$$
On the other hand by $m'$ we have
$$
\tau=2\pi-2\left(2\tan^{-1}\left(m'\tan\frac{\alpha_A+\alpha_B}{2}\right)+\tan^{-1}\left(m'\tan\alpha_C\right)\right)
$$
implying that
$$
\left(2\tan^{-1}\left(m\tan\frac{\alpha_A+\alpha_B}{2}\right)+\tan^{-1}\left(m\tan\alpha_C\right)\right)\leq
$$
$$
\leq \left(2\tan^{-1}\left(m'\tan\frac{\alpha_A+\alpha_B}{2}\right)+\tan^{-1}\left(m'\tan\alpha_C\right)\right)
$$
from which follows that $m'\geq m$.

Second assume that the angle at $C$ is obtuse. Then $\alpha_A+\alpha_B=\alpha_C<\pi/2$ and we have
$$
\tau=2\left(\tan^{-1}\left(m\tan\left(\alpha_A+\alpha_B\right)\right)-\tan^{-1}\left(m\tan\alpha_A\right)-\tan^{-1}\left(m\tan\alpha_B\right)\right).
$$
On the other hand
$$
a(ABC)=\frac{1-m^2}{2}\left(\sin 2\alpha_A+\sin 2\alpha_B-\sin 2\alpha_C\right)
$$
and the volume in question is
\begin{equation}
v(\alpha_A,\alpha_B)=\frac{m(1-m^2)}{6}\left(\sin2\alpha_A+\sin2\alpha_B-\sin 2(\alpha_A+\alpha_B)\right).
\end{equation}
We consider the maximum of $v(\alpha_A,\alpha_B)$ under the conditions $0\leq \alpha_A,\alpha_B\leq \pi/2$,
$$
0=-\frac{\tau}{2}+\left(\tan^{-1}\left(m\tan\left(\alpha_A+\alpha_B\right)\right)- \tan^{-1}\left(m\tan\alpha_A\right)-\tan^{-1}\left(m\tan\alpha_B\right)\right),
$$
and
$$
0=\alpha_A+\alpha_B-\mathrm{const}.
$$
with respect to the unknown values $\alpha_A,\alpha_B$ and $m$. Using Lagrange's method we get two equations
$$
\mu=\frac{m(1-m^2)}{6}\left(\cos 2\alpha_A-\cos 2(\alpha_A+\alpha_B)\right)+\frac{\lambda m(1-m^2)\left(\tan^2\left(\alpha_A+\alpha_B\right)-\tan^2\alpha_A\right) }{\left(1+m^2\tan^2\left(\alpha_A+\alpha_B\right)\right)\left(1+m^2\tan^2\alpha_A\right)}
$$
$$
\mu=\frac{m(1-m^2)}{6}\left(\cos 2\alpha_B-\cos 2(\alpha_A+\alpha_B)\right)+\frac{\lambda m(1-m^2)\left(\tan^2\left(\alpha_A+\alpha_B\right)-\tan^2\alpha_B\right) }{\left(1+m^2\tan^2\left(\alpha_A+\alpha_B\right)\right)\left(1+m^2\tan^2\alpha_B\right)}
$$
which are equivalent to the equations
$$
\frac{\mu}{m(1-m^2)}=\frac{1}{3}+\frac{\lambda\left(1+\tan^2\left(\alpha_A+\alpha_B\right)\right)\left(1+\tan^2\alpha_A\right) }{\left(1+m^2\tan^2\left(\alpha_A+\alpha_B\right)\right)\left(1+m^2\tan^2\alpha_A\right)}
$$
$$
\frac{\mu}{m(1-m^2)}=\frac{1}{3}+\frac{\lambda\left(1+\tan^2\left(\alpha_A+\alpha_B\right)\right)\left(1+\tan^2\alpha_B\right) }{\left(1+m^2\tan^2\left(\alpha_A+\alpha_B\right)\right)\left(1+m^2\tan^2\alpha_B\right)}
$$
because of the equality
$$
\frac{\tan^2\left(\alpha_A+\alpha_B\right)-\tan^2\alpha_A}{\left(1+\tan^2\left(\alpha_A+\alpha_B\right)\right)\left(1+\tan^2\alpha_A\right)}= \cos^2\alpha_A-\cos^2(\alpha_A+\alpha_B)=\frac{\cos 2\alpha_A-\cos 2(\alpha_A+\alpha_B)}{2}.
$$
These conditions turn out to be equivalent to
$$
\frac{\left(1+\tan^2\alpha_A\right)}{\left(1+m^2\tan^2\alpha_A\right)}=\frac{\left(1+\tan^2\alpha_B\right) }{\left(1+m^2\tan^2\alpha_B\right)}
$$
which cannot be satisfied unless $\alpha_A=\alpha_B$. Hence if the triangle is not an isosceles one it is not a local extremum of our problem, on the other hand by compactness it has at least one local maximum proving our statement.
\end{proof}

\begin{remark}
We can compare the formulas of Proposition 2
$$
V\leq\frac{m'(1-m'^2)}{6}\left(2\sin \alpha_C-\sin 2(\alpha_C)\right)=\frac{m'(1-m'^2)}{3}\sin \alpha_C(1-\cos \alpha_C)
$$
and
$$
V\leq \frac{m'(1-m'^2)}{6}\left(2\sin (\pi -\widetilde{\alpha}_C)+\sin 2(\widetilde{\alpha}_C)\right)=\frac{m'(1-m'^2)}{3}\sin \widetilde{\alpha}_C(1+\cos \widetilde{\alpha}_C)
$$
on $\alpha_C$ and $\widetilde{\alpha_C}$. In both cases we assumed that $\alpha_C$ and $\widetilde{\alpha}_C$ are in the interval $[0,\pi/2]$, respectively. Using the equality $\alpha_C=\pi-\widetilde{\alpha}_C$ the above formulas simplify to the following common form
\begin{equation}
V\leq\frac{m'(1-m'^2)}{3}\sin \alpha_C(1-\cos \alpha_C)=:v\left(m',\alpha_C\right) \quad \mbox{where} \quad 0<\alpha<\pi.
\end{equation}
In the case when $AC=BC$ we saw that
$$
\tau=2\left(\tan^{-1}\left(m'\tan\alpha_C\right)-2\tan^{-1}\left(m'\tan\frac{\alpha_C}{2}\right)\right)
$$
and
$$
\tau=2\pi-2\left(2\tan^{-1}\left(m'\tan\frac{\pi-\widetilde{\alpha_C}}{2}\right)+\tan^{-1}\left(m'\tan\widetilde{\alpha_C}\right)\right),
$$
respectively. These equalities can be considered in the following common form
\begin{equation}
\tan\frac{\tau}{2}=\tan\left(\tan^{-1}\left(m'\tan\alpha_C\right)-2\tan^{-1}\left(m'\tan\frac{\alpha_C}{2}\right)\right),
\end{equation}
where $0<\alpha_C<\pi$. In the case when $\pi/2<\alpha_C$ we have $\tan\left(\tau/2\right)<0$ and $\tau/2=\pi+\tan^{-1}\left(\tan\left(\tau/2\right)\right)$.
\end{remark}

\begin{corollary}
The upper bound function for fixed $\tau$ with the parameters $|AB|$,$\alpha_C$ is
\begin{equation}
v(|AB|,\alpha_C):=\frac{|AB|^2}{12}\frac{\sqrt{\sin^2\alpha_C-\frac{|AB|^2}{4}}}{1+\cos\alpha_C},
\end{equation}
and using the equality $|AB|=2\sin \frac{AB}{2}$ it is of the form
\begin{equation}
v(AB, \alpha_C):=\frac{\sin^2\frac{AB}{2}}{3}\frac{\sqrt{\sin^2\alpha_C-\sin^2\frac{AB}{2}}}{1+\cos\alpha_C}.
\end{equation}
If $AB$ is given the maximal volume of the possible facial tetrahedra attained at the isosceles triangle with parameter value $\alpha_C=\cos^{-1}\left(\frac{|AB|^2}{4}-1\right)=\cos^{-1}\left(-\cos^2\frac{AB}{2}\right)$. The formula is
$$
v\left(|AB|,\cos^{-1}\left(\frac{|AB|^2}{4}-1\right)\right)=\frac{|AB|}{6}\sqrt{\left(1-\frac{|AB|^2}{4}\right)}= \frac{1}{6}\sin AB.
$$
\end{corollary}

\begin{proof}
Assume that the value of the length of $AB$ is given. Then by Proposition 2 for fixed $\tau$ the maximal value of the volume $V$ can be attained only for an isosceles triangle and the upper bound function gives this maximal volume.  Using the equality
$$
\sin \alpha_C=\frac{|AB|}{2\sqrt{1-m'^2}}
$$
we get that
$$
v\left(m',\alpha_C\right)=\frac{m'(1-m'^2)}{3}\sin \alpha_C(1-\cos \alpha_C)=\frac{|AB|^2}{12}\frac{\sqrt{\sin^2\alpha_C-\frac{|AB|^2}{4}}}{1+\cos\alpha_C}=v(|AB|,\alpha_C),
$$
where the possible values of $\alpha_C$ can be got from the equality $\sin^2\alpha_C\geq  |AB|^2/4$. The derivative of $v(|AB|,\alpha_C)=v(y,x)$ is
$$
v'(y,x)=\frac{y^2 \sin (x) \sqrt{\sin ^2(x)-\frac{y^2}{4}}}{12 (\cos (x)+1)^2}+\frac{y^2 \sin (x) \cos (x)}{12 (\cos (x)+1) \sqrt{\sin ^2(x)-\frac{y^2}{4}}}
$$
hence we have
$$
v'(|AB|,\alpha_C)=\frac{|AB|^2\sin\alpha_C\left(\cos\alpha_C+1-\frac{|AB|^2}{4}\right)}{12(1+\cos\alpha_C)^2\sqrt{\sin^2\alpha_C-\frac{|AB|^2}{4}}}\quad  \left\{
\begin{array}{ccc}
  <0 & \mbox{ if} & \cos\alpha_C+1<\frac{|AB|^2}{4} \\
  =0 & \mbox{ if} & \cos\alpha_C+1=\frac{|AB|^2}{4}\\
  >0 & \mbox{ if} & \cos\alpha_C+1>\frac{|AB|^2}{4}.
\end{array}
\right.
$$
Since $\cos^{-1}\left(\frac{|AB|^2}{4}-1\right)\leq \pi-\sin^{-1}(|AB|/2)$, on the interval
$$
\sin^{-1}(|AB|/2)<\alpha_C\leq \pi/2\leq \cos^{-1}\left(\frac{|AB|^2}{4}-1\right)\leq \pi-\sin^{-1}(|AB|/2)
$$
the function $v(\alpha_C)$ attains its maximal value at $\cos^{-1}\left(|AB|^2/4-1\right)$ furthermore
$$
v\left(|AB|,\cos^{-1}\left(\frac{|AB|^2}{4}-1\right)\right)=\frac{|AB|^2}{12}\frac{\sqrt{\frac{|AB|^2}{4}\left(1-\frac{|AB|^2}{4}\right)}}{\frac{|AB|^2}{4}} =\frac{|AB|}{6}\sqrt{\left(1-\frac{|AB|^2}{4}\right)}.
$$
$v(|AB|,\alpha_C)$ on the interval $\sin^{-1}(|AB|/2)<\alpha_C \leq \cos^{-1}\left(\frac{|AB|^2}{4}-1\right)$ is a strictly increasing function and on the interval $\cos^{-1}\left(\frac{|AB|^2}{4}-1\right)\leq \pi-\sin^{-1}(|AB|/2)$ it is a decreasing one. This shows that an optimal triangle with the fixed edge length $|AB|$ (which corresponding to a facial tetrahedron with maximal volume) is an isosceles one.
\end{proof}

We also have a formula on the upper bound function $v\left(m',\alpha_C\right)$ using as a parameter the surface area $\tau$ (introduced in Proposition 2).

\begin{prop}
Let the spherical area of the spherical triangle $ABC$ is $\tau$. Let $\alpha_C$ be the greatest central angle of $ABC$ corresponding to $AB$. Then the volume $V$ of the Euclidean pyramid with base $ABC$ and apex $O$ holds the inequality
\begin{equation}
V\leq \frac{1}{3}\tan\frac{\tau}{2}\left(2-\frac{|AB|^2}{4}\left(1+\frac{1}{\left(1+\cos\alpha_C\right)}\right)\right).
\end{equation}
In terms of $\tau $ and $c:=AB$ we have
\begin{equation}
V\leq v(\tau,c):=\frac{1}{6}\sin c\frac{\cos\frac{\tau-c}{2}-\cos\frac{\tau}{2}\cos\frac{c}{2}}{1-\cos\frac{c}{2}\cos\frac{\tau}{2}}.
\end{equation}
Equality holds if and only if $|AC|=|CB|$.
\end{prop}

\begin{proof}
By Proposition 2 and by the note before this statement we have to investigate the inequality
$$
V\leq\frac{m'(1-m'^2)}{3}\sin \alpha_C(1-\cos \alpha_C)=:v\left(m',\alpha_C\right) \quad \mbox{where} \quad 0<\alpha<\pi
$$
with the condition
$$
\tan\frac{\tau}{2}=\tan\left(\tan^{-1}\left(m'\tan\alpha_C\right)-2\tan^{-1}\left(m'\tan\frac{\alpha_C}{2}\right)\right)=
$$
$$
=\frac{m'\tan\alpha_C-\tan\left(2\tan^{-1}\left(m'\tan\frac{\alpha_C}{2}\right)\right)}{1+ m'\tan\alpha_C\tan\left(2\tan^{-1}\left(m'\tan\frac{\alpha_C}{2}\right)\right)}= \frac{\frac{2m'\tan\frac{\alpha_C}{2}}{1-\tan^2\frac{\alpha_C}{2}}-\frac{2m'\tan\frac{\alpha_C}{2}}{1-m'^2\tan^2\frac{\alpha_C}{2}}}{{1+ \frac{2m'\tan\frac{\alpha_C}{2}}{1-\tan^2\frac{\alpha_C}{2}}\frac{2m'\tan\frac{\alpha_C}{2}}{1-m'^2\tan^2\frac{\alpha_C}{2}}}}=
$$
$$
=\frac{2m'(1-m'^2)\tan^3\frac{\alpha_C}{2}}{{(1-\tan^2\frac{\alpha_C}{2})(1-m'^2\tan^2\frac{\alpha_C}{2}})+ 4m'^2\tan^2\frac{\alpha_C}{2}}=
$$
$$
=\frac{2m'(1-m'^2)\tan\frac{\alpha_C}{2}}{(\cot\frac{\alpha_C}{2}-\tan\frac{\alpha_C}{2})(\cot\frac{\alpha_C}{2}-m'^2\tan\frac{\alpha_C}{2})+ 4m'^2}=
$$
$$
=\frac{2m'(1-m'^2)\sin \alpha_C(1-\cos \alpha_C)}{(1-m'^2)\left(\cos\alpha_C-\sin ^2\alpha_C\right)+(1+m'^2)}=\frac{3v\left(m',\alpha_C\right)}{(1-m'^2)\cos\alpha_C(1+\cos\alpha_C)+2m'^2}.
$$
Since
$$
\sin \alpha_C=\frac{|AB|}{2\sqrt{1-m'^2}}
$$
hence
$$
1-m'^2=\frac{|AB|^2}{4\sin^2\alpha_C}
$$
implying that
$$
3v\left(m',\alpha_C\right)= \tan\frac{\tau}{2}\left(\frac{|AB|^2\cos\alpha_C(1+\cos\alpha_C)}{4\sin^2\alpha_C}+2\left(1-\frac{|AB|^2}{4\sin^2\alpha_C}\right)\right)=
$$
$$
= \tan\frac{\tau}{2}\left(2+\frac{|AB|^2}{4\sin^2\alpha_C}\left(\cos\alpha_C(1+\cos\alpha_C)-2\right)\right)=
\tan\frac{\tau}{2}\left(2-\frac{|AB|^2\left(2+\cos\alpha_C\right)}{4\left(1+\cos\alpha_C\right)}\right).
$$
So
$$
V\leq \frac{1}{3}\tan\frac{\tau}{2}\left(2-\frac{|AB|^2}{4}\left(1+\frac{1}{\left(1+\cos\alpha_C\right)}\right)\right)
$$
as we stated.

Since $\pi-\alpha_C$ is the angle of the chordal triangle (rectilineal triangle) $ABC$ at $C$, thus we can give it as a function of the spherical lengths of the sides of the spherical triangle $ABC$. Thus we have (see eq. (486) in \cite{casey})
$$
\cos \alpha_C=-\frac{1+\cos AB-2\cos AC}{4\sin^2\frac{ AC}{2}}=-\frac{-1+\cos AB+4\sin^2 \frac{AC}{2}}{4\sin^2\frac{ AC}{2}}.
$$
Using the notation $a:=BC=AC$, $c=AB$ we get the formula
$$
V\leq \frac{1}{3}\tan\frac{\tau}{2}\left(2-\sin^2\frac{AB}{2}-2\sin^2 \frac{AC}{2}\right)=
$$
$$
=\frac{1}{3}\tan\frac{\tau}{2}\left(2-\sin^2\frac{c}{2}-2\sin^2 \frac{a}{2}\right).
$$
Finally use the spherical Heron's formula proved first by Lhuilier (see p.88 in \cite{casey}):
$$
\tan \frac{\tau}{4}=\sqrt{\tan\frac{a+b+c}{4}\tan\frac{-a+b+c}{4}\tan\frac{a-b+c}{4}\tan\frac{a+b-c}{4}}.
$$
Since $a=b$ it can be reduced to the form
$$
\tan \frac{\tau}{4}=\tan\frac{c}{4}\sqrt{\tan\frac{2a+c}{4}\tan\frac{2a-c}{4}}= \tan\frac{c}{4}\sqrt{\frac{\sin^2\frac{a}{2}-\sin^2\frac{c}{4}}{1-\sin^2\frac{a}{2}-\sin^2\frac{c}{4}}}.
$$
From this we get that
$$
\sin^2\frac{a}{2}=\frac{ \tan^2\frac{\tau}{4}\cos^2\frac{c}{4}+\tan^2\frac{c}{4}\sin^2\frac{c}{4}}{\tan ^2\frac{\tau}{4}+\tan^2\frac{c}{4}}
$$
and thus the inequality
$$
V\leq \frac{1}{3}\tan\frac{\tau}{2}\left(2-\sin^2\frac{c}{2}-2\frac{ \tan^2\frac{\tau}{4}\cos^2\frac{c}{4}+\tan^2\frac{c}{4}\sin^2\frac{c}{4}}{\tan ^2\frac{\tau}{4}+\tan^2\frac{c}{4}} \right)=
$$
$$
=\frac{1}{3}\tan\frac{\tau}{2}\cos\frac{c}{2}\left(\cos\frac{c}{2}+ \frac{\tan^2\frac{c}{4}-\tan^2\frac{\tau}{4}}{\tan^2\frac{c}{4}+\tan^2\frac{\tau}{4}}\right)= \frac{\sin\frac{\tau}{2}\cos\frac{c}{2}\sin^2\frac{c}{2}}{3\left(1-\cos\frac{c}{2}\cos\frac{\tau}{2}\right)}= \frac{\sin c\sin\frac{\tau}{2}\sin\frac{c}{2}}{6\left(1-\cos\frac{c}{2}\cos\frac{\tau}{2}\right)}=
$$
$$
=\frac{1}{6}\sin c\frac{\cos\frac{\tau-c}{2}-\cos\frac{\tau}{2}\cos\frac{c}{2}}{1-\cos\frac{c}{2}\cos\frac{\tau}{2}}.
$$
\end{proof}

\begin{remark}
In the case when $a=b=c$ the connection between the parameters $c$ and $\tau$ is
$$
\tan\frac{\tau}{4}=\tan\frac{c}{4}\sqrt{\tan3\frac{c}{4}\tan\frac{c}{4}}=\tan^2\frac{c}{4}\sqrt{\frac{3-\tan^2\frac{c}{4}}{1-3\tan^2\frac{c}{4}}}.
$$
To determine the parameter $c$ we introduce the notion $x=\tan^2 (c/4)$ and $\theta=\tan^2 (\tau/4)$. Now we get the equation of order three
$$
0=x^3-3x^2-3\theta x +\theta=(x-1)^3-3x(\theta+1)+(\theta+1),
$$
and if we set $y=x-1$ then the equality
$$
0=y^3-3y(\theta+1)-2(\theta+1).
$$
Using Cardano's formula finally we get that
$$
y=\frac{2\cos\left(\frac{\tau}{12}+\frac{4\pi}{3}\right)}{\cos \frac{\tau}{4}}.
$$
Hence we have
$$
\frac{1-\cos\frac{c}{2}}{1+\cos\frac{c}{2}}=\tan^2 \frac{c}{4}=x=\frac{2\cos\left(\frac{\tau}{12}+\frac{4\pi}{3}\right)+\cos \frac{\tau}{4}}{\cos \frac{\tau}{4}}
$$
implying that
$$
\cos\frac{c}{2}=\frac{-1}{2\cos\frac{\tau+4\pi}{6}}
\quad \mbox{ and } \quad
\sin^2 \frac{c}{2}=\frac{4\cos^2\left(\frac{\tau+4\pi}{6}\right)-1}{4\cos^2\left(\frac{\tau+4\pi}{6}\right)}.
$$
Substituting these values into the formula (3.6) we get the inequality (2.2) showing that Proposition 2 is the generalization of Proposition 1 in the case of $p=3$.
\end{remark}

Assume now that the triangular star-shaped with respect to the origin polyhedron $P$ with $f$ face inscribed in the unit sphere.  Let $c_1,\ldots,c_f$ be the arc-lengths of the edges of the faces $F_1,\ldots ,F_f$ corresponding to their maximal central angles, respectively. Denote by $\tau_i$ the spherical area of the spherical triangle corresponding to the face $F_i$ for all $i$. We note that for a spherical triangle which edges $a,b,c$ hold the inequalities $0<a\leq b\leq c<\Pi/2$, also holds the inequality $\tau \leq c$. In fact, for fixed $\tau$ the least value of the maximal edge length attend at the case of regular triangle. If $c<\pi/2$ then we have
$$
\tan\frac{\tau}{4} =\left(\tan\frac{c}{4}\sqrt{\tan3\frac{c}{4}\tan\frac{c}{4}}\right)= \left(\tan\frac{c}{4}\sqrt{1-\frac{\tan\frac{3c}{4}+\tan\frac{c}{4}}{\tan c }}\right)<\tan\frac{c}{4},
$$
and if $c=\pi/2$ then $\tau=8\pi/4=\pi/2$ proving our statement.

Observe that the function $v(\tau,c)$ is concave in the parameter domain $\mathcal{D}:=\{0<\tau<\pi/2, \tau\leq c<\min\{f(\tau),2\sin^{-1}\sqrt{2/3}\}\}$ with certain concave (in $\tau$) function $f(\tau)$ defined by the zeros of the Hessian; and non-concave in the domain $\mathcal{D'}=\{0< \tau \leq \omega, f(\tau)\leq c\leq 2\sin^{-1}\sqrt{2/3}\}=\{0< \tau\leq c\leq \pi/2\}\setminus D$, where $f(\omega)=2\sin^{-1}\sqrt{2/3}$. (The corresponding calculations can be checked by any symbolic software. In the last section we can see such a computation using Mathematica 10. The precise value of $\omega$ can be found in the last section (at Out[22]) which is approximately $\omega\approx 0.697715$.)

\begin{theorem} Assume that $0<\tau_i<\pi/2$ holds for all $i$. For $i=1,\ldots, f'$ we require the inequalities $0<\tau_i\leq c_i\leq \min\{f(\tau_i),2\sin^{-1}\sqrt{2/3}\}$ and for all $j$ with $j\geq f'$ the inequalities $0< f(\tau_j)\leq c_j \leq 2\sin^{-1}\sqrt{2/3}$, respectively. Let denote $c':=\frac{1}{f'}\sum\limits_{i=1}^{f'}c_i$, $c^\star:=\frac{1}{f-f'}\sum\limits_{i=f'+1}^{f}f(\tau_i)$ and $\tau':=\sum\limits_{i=f'+1}^{f}\tau_i$, respectively. Then we have
\begin{equation}
v(P)\leq \frac{f}{6}\sin \left(\frac{f'c'+(f-f')c^\star}{f}\right)\frac{\cos \left(\frac{4\pi-f'c'-(f-f')c^\star}{2f}\right)-\cos\frac{2\pi}{f}\cos\left(\frac{f'c'+(f-f')c^\star}{2f}\right)} {1-\cos\frac{4\pi}{2f}\cos\left(\frac{f'c'+(f-f')c^\star}{2f}\right)} .
\end{equation}
\end{theorem}

\begin{proof}
The volume of $P$ is bounded above by the quantity
$$
v(P)\leq\sum\limits_{i=1}^fv(\tau_i,c_i):=\frac{1}{6}\sum\limits_{i=1}^f\sin c_i\frac{\cos\frac{\tau_i-c_i}{2}-\cos\frac{\tau_i}{2}\cos\frac{c_i}{2}}{1-\cos\frac{c_i}{2}\cos\frac{\tau_i}{2}}.
$$
Using the concavity of the function $v(\tau,c)$ on the domain $\mathcal{D}$ and the fact that the function $v(\tau,\cdot)$ for fixed $\tau$ is monotone decreasing by $c$ on the domain $\mathcal{D'}$, we get the following upper bound for $v(P)$:
$$
v(P)\leq  \frac{f'}{6}v\left(\frac{4\pi-\tau'}{f'},c'\right)+\frac{f-f'}{6}v\left(\frac{\tau'}{f-f'},c^\star\right).
$$
Since for $i=f'+1,\ldots,f$ the points $(\tau_i,f(\tau_i))$ are in the convex domain $D$ then the point $\left(\frac{\tau'}{f-f'},c^\star\right)$ also in $\mathcal{D}$. Applying again the concavity property of the function $v(\tau,c)$, we get the inequality
$$
v(P)\leq \frac{f}{6}v\left(\frac{4\pi}{f},\frac{f'c'+(f-f')c^\star}{f}\right)=
$$
$$
=\frac{f}{6}\sin \left(\frac{f'c'+(f-f')c^\star}{f}\right)\frac{\cos \left(\frac{4\pi-f'c'-(f-f')c^\star}{2f}\right)-\cos\frac{2\pi}{f}\cos\left(\frac{f'c'+(f-f')c^\star}{2f}\right)} {1-\cos\frac{4\pi}{2f}\cos\left(\frac{f'c'+(f-f')c^\star}{2f}\right)},
$$
as we stated.
\end{proof}

\begin{remark}
When $f'=f$ we have the following formula:
\begin{equation}
v(P)\leq \frac{f}{6}\sin c'\frac{\cos\left(\frac{2\pi}{f}-\frac{c'}{2}\right)-\cos\frac{2\pi }{f}\cos\frac{c'}{2}}{1-\cos\frac{c'}{2}\cos\frac{2\pi}{f}},
\end{equation}
where $c'=\frac{1}{f}\sum\limits_{i=1}^{f}c_i$. In this case the upper bound is sharp if all face-triangles are obtuse isosceles ones with the same area and maximal edge lengths.
\end{remark}

The condition of sharpness implies that the unit sphere tiling by the congruent copies of such isosceles spherical triangles which equal sides are less than or equal to the third one. Observe that a polyhedron corresponding to such a tiling could not be convex. This motivates the following problem:
\emph{Give such values $\tau$ and $c$ that the isosceles spherical triangle with area $\tau $ and unique maximal edge length $c$ can be generate a tiling of the unit sphere.}
We note that the regular bodies with triangular faces hold the above property.

\begin{example}
We get a non-trivial example for this question, if we consider a rhombic dodecahedron with the same centroid as of the center of the sphere and we project from the center its vertices to the sphere (see the left figure in Fig.2). (Note that there is no circumscribed sphere about a rhombic dodecahedron hence the projection is necessary.)  We get a tiling of the sphere containing congruent spherical quadrangles. One of these quadrangles has four congruent sides and two diagonals, respectively. The length of the longer diagonal is $c=\pi/2$.

\begin{figure}[ht]
\includegraphics[scale=1]{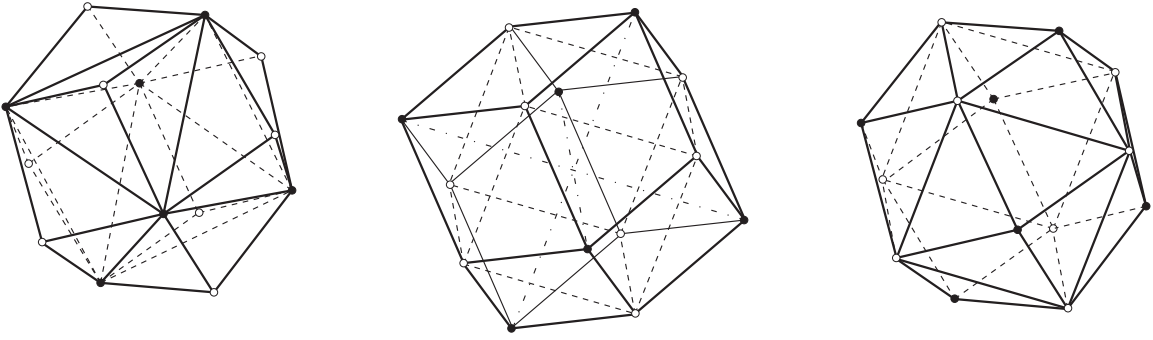}
\caption[]{The star-shaped polyhedron $P$ (on left), the original rhombic dodecahedron and the convex convex hull $Q$ of $P$ (on right).}
\end{figure}

We can dissect these quadrangles at this longer diagonals into  two congruent spherical triangles. Denote by $\mathcal{P}$ the polyhedron defined by those plane triangles as facets which corresponds to these spherical triangles, respectively. The angles and sides have the respective measures $\gamma=2\pi/3,\alpha=\pi/4,\beta=\pi/4$ and $c=\pi/2,a=\sin^{-1}\sqrt{2/3},b=\sin^{-1}\sqrt{2/3}$. Hence the area of this triangle $2\pi/3+\pi/2+\pi/2-\pi=\pi/6=4\pi/24$ as follows from the fact, that the $24$ congruent copies of it, tile the whole sphere. Observe that $\mathcal{P}$ is not convex since the distance of the opposite vertices of two triangles with common base (in Euclidean measure) ($2/\sqrt{3}$)  is less than that of the Euclidean length of the common base ($\sqrt{2}$).
Since we have only one type of triangles for which $f(\tau_1)=f(\pi/6)=f(0,52360)\geq \pi/2=c_1$ we can apply (3.9) with $f=24$, $c'=\pi/2$, hence
$$
v(\mathcal{P})=4\frac{\sqrt{2}\cos\frac{\pi}{6}-\cos\frac{\pi}{12}}{\sqrt{2}-\cos\frac{\pi}{12}}.
$$
This quantity is an upper bound for the volume of such star-shaped polyhedra which inscribed into the unit sphere, have 24 faces with spherical area $\tau_i$ with the assumption $f(\tau_i)\geq \pi/2$ and with maximal edge length $\pi/2$. We get such polyhedra if we change a little-bit the position of those vertices of $P$ which denoted by white circles on Fig.2. (For $\tau$ (by Mathematica 10 see in the last section at In[20]) we get the assumption $\pi/2\geq \tau\geq \tan^{-1}\left(2 \sqrt{5}-3 \sqrt{2}\right)/(10+7 \sqrt{2})\approx \tau=0.427922$.)

Denote by $Q$ the convex hull of $P$ (see the right figure on Fig.2). Then $c_1=2\sin^{-1}\sqrt{1/3}\approx 1,23096<\pi/2<f(\tau_1)$ and we can apply again (3.9). Hence we get that
$$
v(\mathcal{Q})=\frac{8}{3}\frac{\sqrt{6}\cos\left(\frac{\pi}{12}-\sin^{-1}\sqrt{\frac{1}{3}}\right)-2\cos\frac{\pi}{12}}{\sqrt{3}-\cos\frac{\pi}{12}\sqrt{2}}.
$$
$Q$ has maximal volume of the class of such polyhedra which can be gotten from $Q$ by a little change of the position of the vertices denoted by black circles, respectively.
\end{example}

\begin{example}
Assume that $f'=f=12$ and $c=2\sin^{-1}(\sqrt{2/3})$. Then the upper bound is
$$
2\frac{2\sqrt{2}}{3}\frac{\cos\left(\frac{\pi}{6}-\sin^{-1}(\sqrt{2/3})\right)-\frac{1}{\sqrt{3}}\cos\frac{\pi }{6}}{1-\frac{1}{\sqrt{3}}\cos\frac{\pi}{6}}=\frac{8}{3\sqrt{3}},
$$
which is the volume of the cube inscribed into the unit sphere. Hence we got a new proof for that case of Theorem 3.3 of \cite{gho 1} when we restrict our examination to those triangulations in which there is no face-triangle having edge length greater than the edge length of a regular tetrahedron inscribed into the unit sphere.
\end{example}

We now apply our inequality (3.7) to prove the general form of Theorem 3.3 in \cite{gho 1} in which the additional assumption "the tetrahedra are in dual position" has been omitted.

\begin{theorem}
Consider two regular tetrahedron inscribed into the unit sphere. The maximal volume of the convex hull $P$ of the eight vertices is the volume of the cube $C$ inscribed into to unit sphere, so
$$
v(P)\leq v(C)=\frac{8}{3\sqrt{3}}.
$$
\end{theorem}

\begin{proof}
We have to consider only that case which is not was considered in \cite{gho 1}. Hence we assume that in the spherical regular triangles of the spherical tiling corresponding to the first regular tetrahedron there are $2,1,1,0$ vertices of the second tetrahedron, respectively. The five points  (the three vertices of the first spherical triangle and the two vertices of the second tetrahedron having in this triangle) having in the first closed spherical triangle form a triangular dissection of it into five other spherical triangle. Unfortunately, this dissection contains also such triangles which maximal edge lengths greater than that of the edge length of the regular spherical triangle contains them. On the other hand these triangles belongs to the parameter domain $\mathcal{D}'$ (defined in Theorem 1) because $f(\pi/5)=1.83487<2\sin^{-1}\sqrt{\frac{2}{3}}$. Hence the upper bound function for fixed $\tau$ is locally an decreasing function of $c$. So we can assume that all of these triangles have the same maximal spherical lengths are equal to $2\sin^{-1}\sqrt{\frac{2}{3}}$. Thus we get the following upper bound for the volume:
$$
v(P)\leq v\left(\pi,2\sin^{-1}\sqrt{\frac{2}{3}}\right)+6v\left(\pi/3,2\sin^{-1}\sqrt{\frac{2}{3}}\right)+ \sum\limits_{i=1}^{5}v\left(\tau_i,2\sin^{-1}\sqrt{\frac{2}{3}}\right)=
$$
$$
=\frac{1}{9}+\frac{4}{3\sqrt{3}}+ \frac{2}{9}\sum\limits_{i=1}^{5}\frac{\sin\frac{\tau_i}{2}}{\sqrt{3}-\cos\frac{\tau_i}{2}}
$$
where $0\leq \tau_i$ and $\sum\limits_{i=1}^{5}\tau_i=\pi$. But with this conditions as we can check in Section 5 (at In[30])
$$
\sum\limits_{i=1}^{5}\frac{\sin\frac{\tau_i}{2}}{\sqrt{3}-\cos\frac{\tau_i}{2}}\leq 1.97836<2
$$
implying that
$$
v(P)< \frac{\frac{1}{\sqrt{3}}+4+\frac{4}{\sqrt{3}}}{3\sqrt{3}}<\frac{8}{3\sqrt{3}}=v(C)
$$
as we stated.
\end{proof}

\section{Notes on the general case of the inequality}

In this section we assume that the examined face $F$ of the polyhedron $P$ has $p$ sides. Then we have $p$ points of a circle with center $K$ of the unit sphere which corresponds to a spherical convex polygon and also an rectilineal polygon with the same set of vertices $\{A_1,\ldots, A_p\}$, respectively. We introduce the concept of \emph{central angle} $\alpha_{i,i+1}$ which is the half of the convex angle $A_iKA_{i+1}\measuredangle$ when the line $A_iA_{i+1}$ does not separate $K$ to the other vertices of the polygon, and it is the half of the value $2(\pi-A_iKA_{i+1}\measuredangle)$ in the other case. Then the angle $0<\alpha_{i,i+1}<\pi$ corresponds to the edge $A_iA_{i+1}$. Analogously, we can introduce the angles $\beta_{i,i+1}$ which is the angle of the spherical triangle $A_iA_{i+1}K'$ at the vertex $A_i$ (and $A_{i+1}$), where $K'$ is the central projection of $K$ to the sphere. Now the proof of Prop. 2 can be modified as follows.
If the rectilineal polygon contains $K$ then
$$
a(KA_iA_{i+1})=\Delta(\alpha_{i,i+1},\beta_{i,i+1})=\frac{1}{2}\sin 2\alpha_{i,i+1}\sin ^2 K'OA_i\angle=
$$
$$
=\frac{1}{2}\sin 2\alpha_{i,i+1}\left(1-\cot^2\alpha_{i,i+1}\cot^2\beta_{i,i+1}\right).
$$
On the domain
$$
0\leq \alpha_{i,i+1} \leq \frac{\pi}{2}, \quad 0\leq \beta_{i,i+1} \leq \frac{\pi}{2}, \quad \alpha_{i,i+1}+\beta_{i,i+1} \geq \frac{\pi}{2}
$$
it is a concave function of two variables (see p.267 in \cite{ftl}). Hence
$$
a(A_1\ldots A_p)=\sum\limits_{i=1}^{p-1}\Delta(\alpha_{i,i+1},\beta_{i,i+1})+\Delta(\alpha_{p,1},\beta_{p,1})\leq
$$
$$
\leq (p-1)\Delta\left(\frac{\sum\limits_{i=1}^{p-1}\alpha_{i,i+1}}{p-1}, \frac{\sum\limits_{i=1}^{p-1}\beta_{i,i+1}}{p-1}\right)+\Delta(\alpha_{p,1},\beta_{p,1})=a(A_1'\ldots A_{p}'),
$$
where the value of the right hand side is the area of a polygon $A_{1}'\ldots A_{p}'$ whose edges $A_i'A_{i+1}'$ have equal lengths for $i=1,\ldots,p-1$. (We note that the central projections upon the sphere of the two polygons have the same spherical excess $a(A_1\ldots A_{p})=a(A_1'\ldots A_{p-1}')=2\sum_{i=1}^{p}\beta_{i,i+1}-(p-2)\pi$.)

Compare now the altitudes $m$ and $m'$ of the pyramids based on the two polygons, respectively. The spherical area of the first one is
$$
\tau=2\sum\limits_{i=1}^{p}\beta_{i,i+1}-(p-2)\pi=2\pi+\left(2\sum\limits_{i=1}^p\beta_{i,i+1}-p\pi\right)=2\pi+
$$
$$
+2\left(\sum\limits_{i=1}^{p-1}\tan^{-1}\left(\tan\left(\beta_{i,i+1}-\frac{\pi}{2}\right)\right)+ \tan^{-1}\left(\tan\left(\beta_{p,1}-\frac{\pi}{2}\right)\right)\right)=
$$
$$
=2\pi-2\left(\sum\limits_{i=1}^{p-1}\tan^{-1}\left(m\tan\alpha_{i,i+1}\right)+\tan^{-1}\left(m\tan\alpha_{p,1}\right)\right).
$$
We get that
$$
\tau\leq 2\pi-2\left((p-1)\tan^{-1}\left(m\tan\frac{\sum\limits_{i=1}^{p-1}\alpha_{i,i+1}}{p-1}\right)+\tan^{-1}\left(m\tan\alpha_{p,1}\right)\right).
$$
On the other hand by $m'$ we have
$$
\tau=2\pi-2\left((p-1)\tan^{-1}\left(m'\tan\frac{\sum\limits_{i=1}^{p-1}\alpha_{i,i+1}}{p-1}\right)+\tan^{-1}\left(m'\tan\alpha_{p,1}\right)\right)
$$
implying that $m'\geq m$.

We have to investigate also the second case, when $A_pA_1$ separates $K$ and the other vertices to each other. Then $\sum_{i=1}^{p-1}\alpha_{i,i+1}=\pi-\alpha_{p,1}=\widetilde{\alpha_{p,1}}<\pi/2$ and we have
$$
\tau=2\left(\tan^{-1}\left(m\tan\left(\sum\limits_{i=1}^{p-1}\alpha_{i,i+1}\right)\right)- \sum\limits_{i=1}^{p-1}\tan^{-1}\left(m\tan\alpha_{i,i+1}\right)\right).
$$
On the other hand
$$
a(A_1\ldots A_p)=\frac{1-m^2}{2}\left(\sum\limits_{i=1}^{p-1}\sin 2\alpha_{i,i+1}-\sin 2\widetilde{\alpha_{p,1}}\right)
$$
and the volume in question is
\begin{equation}
v(\alpha_{1,2},\ldots,\alpha_{p-1,p})=\frac{m(1-m^2)}{6}\left(\sum\limits_{i=1}^{p-1}\sin2\alpha_{i,i+1}-\sin 2\sum\limits_{i=1}^{p-1}\alpha_{i,i+1}\right).
\end{equation}
We consider the maximum of $v(\alpha_{1,2},\ldots,\alpha_{p-1,p})$ under the conditions $0\leq \alpha_{i,i+1}\leq \pi/2$,
$$
0=-\frac{\tau}{2}+\left(\tan^{-1}\left(m\tan\left(\sum\limits_{i=1}^{p-1}\alpha_{i,i+1}\right)\right)- \sum\limits_{i=1}^{p-1}\tan^{-1}\left(m\tan\alpha_{i,i+1}\right)\right),
$$
and
$$
0=\sum\limits_{i=1}^{p-1}\alpha_{i,i+1}-\mathrm{const}.
$$
with respect to the unknown values $\alpha_{i,i+1}$ and $m$. As in the proof of Proposition 2, using Lagrange's method we get $p-1$ equations of form
$$
\mu=\frac{m(1-m^2)}{6}\left(\cos 2\alpha_{i,i+1}-\cos 2\sum\limits_{i=1}^{p-1}\alpha_{i,i+1}\right)+\frac{\lambda_i m(1-m^2)\left(\tan^2\sum\limits_{i=1}^{p-1}\alpha_{i,i+1}-\tan^2\alpha_{i,i+1}\right) }{\left(1+m^2\tan^2\sum\limits_{i=1}^{p-1}\alpha_{i,i+1}\right)\left(1+m^2\tan^2\alpha_{i,i+1}\right)}.
$$
These leads to the equations
$$
\frac{\left(1+\tan^2\alpha_{1,2}\right)}{\left(1+m^2\tan^2\alpha_{1,2}\right)}=\ldots=\frac{\left(1+\tan^2\alpha_{p-1,p}\right) }{\left(1+m^2\tan^2\alpha_{p-1,p}\right)}
$$
which cannot be satisfied unless $\alpha_{1,2}=\ldots =\alpha_{p-1,p}$. Hence if the corresponding sides of the polygon are not equals then the polygon is not a local extremum of our problem. On the other hand, by compactness, it has at least one local maximum proving the following statement:

\begin{prop}
Let $A_1\ldots A_p$ be a convex polygon inscribed into a circle of the unit sphere. Then there is another convex polygon $A_1'\ldots A_p'$ inscribed also into a circle of the unit sphere with the following properties:
\begin{itemize}
\item $A_1'A_{2}'=\ldots =A_{p-1}'A_{p}'$;
\item the greatest central angles and also the spherical areas of the two polygons are equal to each other, respectively;
\item the volume of the facial pyramid with base $A_1'\ldots A_{p}'$ and apex $O$ is greater than or equal to the volume of the pyramid with base $A_1\ldots A_p$ and apex $O$.
\end{itemize}
\end{prop}

We can rewrite the formulas (3.2), (3.4), and (3.5) of Section 3 to the corresponding formulas on the upper bound functions of a facial pyramid based on a $p$-gon, respectively. We get that
\begin{equation}
v(m',\alpha_{p,1})=\frac{m'(1-m'^2)}{6}\left((p-1)\sin \frac{2\alpha_{p,1}}{p-1}-\sin 2\alpha_{p,1}\right).
\end{equation}
Since
$$
\sin \alpha_{p,1}=\frac{|A_1A_p|}{2\sqrt{1-m'^2}}, \mbox{ hence } m'(1-m'^2)=\frac{|A_1A_p|^2\sqrt{\sin ^2\alpha_{p,1}-\frac{|A_1A_p|^2}{4}}}{4\sin ^3\alpha_{p,1}}
$$
we have that
\begin{equation}
v(|A_1A_p|,\alpha_{p,1}):=\frac{|A_1A_p|^2\sqrt{\sin^2\alpha_{p,1}-\frac{|A_1A_p|^2}{4}}}{24}\frac{\left((p-1)\sin \frac{2\alpha_{p,1}}{p-1}-\sin 2\alpha_{p,1}\right)}{\sin^3\alpha_{p,1}},
\end{equation}
and using the equality $|A_1A_p|=2\sin \frac{A_1A_p}{2}$ it gets the form
\begin{equation}
v(A_1A_p, \alpha_{p,1}):= \frac{\sin^2\frac{A_1A_p}{2}\sqrt{\sin^2\alpha_{p,1}-\sin^2\frac{A_1A_p}{2}}}{6}\frac{\left((p-1)\sin \frac{2\alpha_{p,1}}{p-1}-\sin 2\alpha_{p,1}\right)}{\sin^3\alpha_{p,1}}.
\end{equation}
Finally, we note that the formulas (3.6) and (3.7) cannot be reproduced in an obvious way. The reason is that a complicated trigonometric formula
$$
\tan\frac{\tau}{2}=\tan\left(\tan^{-1}\left(m'\tan\alpha_{p,1}\right)- (p-1)\tan^{-1}\left(m'\tan\frac{\alpha_{p,1}}{p-1}\right)\right),
$$
$$
=\frac{m'\tan\alpha_{p,1}-\tan\left((p-1)\tan^{-1}\left(m'\tan\frac{\alpha_{p,1}}{p-1}\right)\right)}{1+ m'\tan\alpha_{p,1}\tan\left((p-1)\tan^{-1}\left(m'\tan\frac{\alpha_{p,1}}{p-1}\right)\right)}
$$
connects the parameters $\tau,m'$ and $\frac{\alpha_{p,1}}{p-1}$ excluding the possibility to get explicit formulas similar to (3.6) and (3.7), respectively.

\section{Numerical and symbolic computations with Mathematica 10}

Present section contains those command lines (in Mathematica 10), which support our statements on the examined functions. The source file with the results can be found on the web \cite{gho 3}. 

\noindent {\bf In[1]:} \quad
\(D[1/6\text{Sin}[c](\text{Cos}[(x-c)/2]-\text{Cos}[x/2]\text{Cos}[c/2])/(1-\text{Cos}[x/2]\text{Cos}[c/2]), x,c]\)

\noindent {\bf In[2]:} \quad 
\(\text{Simplify}[\%]\)

\noindent {\bf In[3]:} \quad 
\(D[1/6\text{Sin}[c](\text{Cos}[(x-c)/2]-\text{Cos}[x/2]\text{Cos}[c/2])/(1-\text{Cos}[x/2]\text{Cos}[c/2]), \{x,2\}]
\)

\noindent {\bf In[4]:} \quad 
\(\text{Simplify}[\%]\)

\noindent {\bf In[5]:} \quad 
\(D[1/6\text{Sin}[c](\text{Cos}[(x-c)/2]-\text{Cos}[x/2]\text{Cos}[c/2])/(1-\text{Cos}[x/2]\text{Cos}[c/2]), \{c,2\}]
\)

\noindent {\bf In[6]:} \quad 
\(\text{Simplify}[\%]
\)

\noindent {\bf In[7]:} \quad 
\(
D[1/6\text{Sin}[c](\text{Cos}[(x-c)/2]-\text{Cos}[x/2]\text{Cos}[c/2])/(1-\text{Cos}[x/2]\text{Cos}[c/2]), \{x,2\}]D[1/6\text{Sin}[c] \\
\text{Cos}[(x-c)/2]-\text{Cos}[x/2]\text{Cos}[c/2])/(1-\text{Cos}[x/2]\text{Cos}[c/2]),
\{c,2\}]-D[1/6\text{Sin}[c](\text{Cos}[(x-c)/2]- \\
\text{Cos}[x/2]\text{Cos}[c/2])/(1-\text{Cos}[x/2]\text{Cos}[c/2]), x,c]{}^{\wedge}2
\)

\noindent {\bf In[8]:} \quad 
\(
\text{Simplify}[\%]
\)

\noindent {\bf In[9]:} \quad 
\(
\text{Plot3D}[1/6\text{Sin}[c](\text{Cos}[(x-c)/2]-\text{Cos}[x/2]\text{Cos}[c/2])/(1-\text{Cos}[x/2]\text{Cos}[c/2]), \{x,0,\text{Pi}/2\},\\ \{c,x,2\text{ArcSin}[\text{Sqrt}[2/3]]\}]
\)

\noindent {\bf In[10]:} \quad 
\(
D[1/6\text{Sin}[c](\text{Cos}[(x-c)/2]-\text{Cos}[x/2]\text{Cos}[c/2])/(1-\text{Cos}[x/2]\text{Cos}[c/2]), c]
\)

\noindent {\bf In[11]:} \quad 
\(
\text{Simplify}[\%]
\)

\noindent {\bf In[12]:} \quad 
\(
\text{Plot3D}[\text{Out}[11], \{x,0,\text{Pi}/2\}, \{c,x,2\text{ArcSin}[\text{Sqrt}[2/3]]\}]
\)

\noindent {\bf In[13]:} \quad  
\(
\text{RegionPlot}[\text{Out}[11]\text{$>$=}0 ,\{x,0,\text{Pi}/2\},\{c,0,2\text{ArcSin}[\text{Sqrt}[2/3]]\}]
\)

\noindent {\bf In[14]:} \quad 
\(
\text{Plot3D}[\text{Out}[8], \{x,0,\text{Pi}/2\}, \{c,x,2\text{ArcSin}[\text{Sqrt}[2/3]]\}]
\)

\noindent {\bf In[15]:} \quad 
\(
\text{RegionPlot}[\text{Out}[8]\text{$>$=}0 ,\{x,0,\text{Pi}/2\},\{c,0,2\text{ArcSin}[\text{Sqrt}[2/3]]\}]
\)

\noindent {\bf In[16]:} \quad 
\(
\text{Reduce}[ \text{Out}[8]\text{==}0 \&\& x\text{==}\text{Pi}/2 \&\& 0<c<2\text{ArcSin}[\text{Sqrt}[2/3]], \{x,c\}]
\)

\noindent {\bf In[17]:} \quad 
\(
N[\%]
\)

\noindent {\bf In[18]: } \quad 
\(
\text{Solve}[z{}^{\wedge}4-24z{}^{\wedge}3+78z{}^{\wedge}2-24z+1\text{==}0]
\)

\noindent {\bf In[19]: } \quad 
\(N[\%]\)

\noindent {\bf In[20]: } \quad 
\(
\text{Reduce}[ \text{Out}[8]\text{==}0 \&\& c\text{==}\text{Pi}/2 \&\& $0<x<\text{Pi}/2$, \{x,c\}]
\)

\noindent {\bf In[21]:} \quad
\( N[\%] \)

\noindent {\bf In[22]:} \quad 
\(
\text{Reduce}[ \text{Out}[8]\text{==}0 \&\& c\text{==}2\text{ArcSin}[\text{Sqrt}[2/3]] \&\& 0<x<\text{Pi}/2, \{x,c\}]
\)

\noindent {\bf In[23]:} \quad 
\(
N[\%]
\)

\noindent {\bf In[24]:} \quad 
\(
\text{Reduce}[ \text{Out}[8]\text{==}0 \&\& x\text{==}-2\text{Pi}/3+2\text{ArcSin}[\text{Sqrt}[14]/4] \&\& x<c<2\text{ArcSin}[\text{Sqrt}[2/3]],
\{x,c\}]
\)

\noindent {\bf In[25]:} \quad
\(
N[\%]
\)

\noindent {\bf In[26]:} \quad 
\(
\text{Reduce}[ \text{Out}[8]\text{==}0 \&\& x\text{==}-4\text{ArcSin}[\text{Sqrt}[14]/4]+5\text{Pi}/3 \&\& x<c<\text{Pi}/2, \{x,c\}]
\)

\noindent {\bf In[27]:} \quad 
\(
N[\%]
\)

\noindent {\bf In[28]:} \quad 
\(
\text{Reduce}[ \text{Out}[8]\text{==}0 \&\& x\text{==}\text{Pi}/5 \&\& 0<c<2\text{ArcSin}[\text{Sqrt}[2/3]], \{x,c\}]
\)

\noindent {\bf In[29]:} \quad 
\(
N[\%]
\)

\noindent {\bf In[30]:} \quad 
\(
\text{NMaximize}[\{\text{Sin}[x]/(\text{Sqrt}[3]-\text{Cos}[x])+\text{Sin}[y]/(\text{Sqrt}[3]-\text{Cos}[y])+\text{Sin}[z]/(\text{Sqrt}[3]-\text{Cos}[z])+\text{Sin}[u]/(\text{Sqrt}[3]-\text{Cos}[u])+\text{Sin}[v]/(\text{Sqrt}[3]-\text{Cos}[v]),
x+y+z+v+u\text{$<$=}\text{Pi}/2, 0<x, 0<y, 0<z, 0<u, 0<v\}, \{x,y,z,u,v\}]
\)

\end{document}